\documentclass[12pt,reqno]{amsart}
\usepackage{amssymb,amsthm,amsmath,amstext,amsxtra}
\usepackage{mathrsfs,bm}
\usepackage{fullpage}
\usepackage[all]{xy}
\usepackage{booktabs}
\usepackage{mathtools}
\usepackage{hyperref}
\hypersetup{colorlinks=true,urlcolor=blue,citecolor=blue,linkcolor=blue}
\usepackage{enumerate}
\usepackage{enumitem}
\usepackage{comment}
\usepackage{multirow}
\usepackage{colonequals}
\usepackage{tikz}
\usepackage{marginnote}
\usepackage[export]{adjustbox}

\usepackage{tikz}
\usepackage{tikz-cd}
\usepackage{xcolor}

\numberwithin{equation}{section}

\newtheorem{theorem}{Theorem}[section]
\newtheorem{lemma}[theorem]{Lemma}

\newtheorem{remark}[theorem]{Remark}
\newtheorem{definition}[theorem]{Definition}

\begin{document}

\title{The Gamma conjecture for $G$-functions}

\author{Wenzhe Yang}
\address{SITP Stanford University, CA, 94305}
\email{yangwz@stanford.edu}

\begin{abstract}
The Bombieri-Dwork conjecture predicts that the differential equations satisfied by $G$-functions come from geometry. In this paper, we will look at special $G$-functions whose differential equations have a special singularity with maximally unipotent monodromy. We will formulate a Gamma conjecture about such $G$-functions, which has close connections with the mirror symmetry of Calabi-Yau threefolds and the Gamma conjecture in algebraic geometry. We will provide examples to support this conjecture, which involves numerical computations using Mathematica programs.
\end{abstract}

\maketitle
\setcounter{tocdepth}{1}
\vspace{-13pt}
\vspace{-13pt}

\vspace*{0.2in}


\tableofcontents
\section{Introduction}

The Picard-Fuchs equations of families of algebraic varieties have played a crucial role in the developments of many area of mathematics, e.g. algebraic geometry, number theory etc. We will intuitively say that these differential equations come from geometry, while the Bombieri-Dwork conjecture is an attempt to characterize which differential equations are from geometry. 

Let us now explain the Bombieri-Dwork conjecture using an interesting example. In Ap\'ery's proof of the irrationality of $\zeta(3)$, the following series sum
\begin{equation} \label{eq:zeta3powerseries}
\zeta(3)=\frac{5}{2}\, \sum_{n=1}^{\infty} \frac{(-1)^{n-1}}{n^3 \binom{2n}{n}}
\end{equation}
has played a crucial role. We now construct a power series $\Pi_0(\phi)$ by
\begin{equation}
\Pi_0(\phi)=2\,\sum_{n=1}^\infty  \frac{(-1)^{n-1}}{n^3 \binom{2n}{n}} \,\phi^{n-1},
\end{equation}
which converges on the disc $|\phi|<4$. The denominator of the coefficients of $\Pi_0(\phi)$, i.e. $n^3 \binom{2n}{n}$, does not grow too fast. More precisely, there exist a constant $C$ such that
\begin{equation}
n^3 \binom{2n}{n} \leq C^n.
\end{equation}
Furthermore, $\Pi_0(\phi)$ satisfies a fourth order ordinary differential equation of the form
\begin{equation} \label{eq:picardfuchsphi}
\left( \left(4+\phi \right)\vartheta^4+\left(10+4\phi \right) \vartheta^3+\left(8+6\phi \right) \vartheta^2+\left(2+4\phi \right)\vartheta+\phi \right) \Pi_0(\phi)=0,~\vartheta=\phi \frac{d}{d\phi}.
\end{equation}
By abuse of notations, we will call this ODE the Picard-Fuchs equation of $\Pi_0(\phi)$. Such a series $\Pi_0(\phi)$ is an example of $G$-functions. The Bombieri-Dwork conjecture predicts that the differential equation \ref{eq:picardfuchsphi} comes from geometry, however it is much more difficult to prove this claim. In this paper, we will study this Picard-Fuchs equation from another point of view, which we briefly explain now!

The Picard-Fuchs equation \ref{eq:picardfuchsphi} has three regular singularities at
\begin{equation}
\phi=0,-4,\infty,
\end{equation}
while the monodromy at $\phi=\infty$ is maximally unipotent. For simplicity, let us define $\varphi$ by
\begin{equation}
\varphi:=1/\phi.
\end{equation}
In a small neighborhood of $\varphi=0$, the Picard-Fuchs equation \ref{eq:picardfuchsphi} has four canonical solutions of the form
\begin{equation}
\begin{aligned}
\varpi_0(\varphi)&=\varphi, \\
\varpi_1(\varphi)&=\varphi \log \varphi,\\
\varpi_2(\varphi)&=\varphi \log^2 \varphi+h_2(\varphi),\\
\varpi_3(\varphi)&=\varphi \log^3 \varphi+3h_2(\varphi) \log \varphi+h_3(\varphi),
\end{aligned}
\end{equation}
where $h_2(\varphi)$ and $h_3(\varphi)$ are power series in $\varphi$ that converge in a small neighborhood of $\varphi=0$. Moreover, they satisfy
\begin{equation}
h_2(0)=h_3(0)=0.
\end{equation}
The four canonical solutions $\{\varpi_i(\varphi) \}_{i=0}^3$ form a basis of the solution space of  \ref{eq:picardfuchsphi}, and $\Pi_0(\phi)$ must have an expansion of the form
\begin{equation}
\Pi_0(\phi)=\tau_0\,\varpi_0(\varphi)+\tau_1 \,\varpi_1(\varphi)+\tau_2 \,\varpi_2(\varphi)+\tau_3\,\varpi_3(\varphi),~\tau_i \in \mathbb{C}.
\end{equation}
However the complex numbers $\tau_i$ are very difficult to determine by analytic methods.

Since the Bombieri-Dwork conjecture has not been proved, we do not know whether the differential equation \ref{eq:picardfuchsphi} comes from geometry. But we can still ask whether $\tau_i$ are ``interesting'' numbers? If we assume the Bombieri-Dwork conjecture, then from the mirror symmetry of Calabi-Yau threefolds, we would expect $\tau_i$ to be zeta values \cite{Galkin, KimYang, Yang}. In this paper, we will resort to a numerical approach to compute $\tau_i$, and we have found that
\begin{equation} \label{eq:introductionpi0trans}
\tau_0=4\,\zeta(3),~\tau_1=0,~\tau_2=0,~\tau_3=-\frac{1}{3},
\end{equation}
which exactly agrees with our expectation. We will give more examples in this paper, all of which have shown the occurrence of zeta values. To explain this phenomenon, we will formulate a Gamma conjecture for the $G$-functions whose Picard-Fuchs equations have a singularity with maximally unipotent monodromy. The formula \ref{eq:introductionpi0trans} has one more interesting application, and it will give us an identity
\begin{equation}
\sum_{n=1}^\infty \frac{\binom{2n}{n}}{4^n n^2}=\frac{1}{6}\, \pi^2-2\,\log^2 2.
\end{equation}
Furthermore, other examples also yield equations like
\begin{equation}
\sum_{n=1}^\infty \frac{\binom{2n}{n}}{4^n n^5}=6\,\zeta(5)-\frac{1}{20} \pi^4 \log 2-\frac{1}{3} \zeta(3) \pi^2+4 \, \zeta(3) \, \log^2 2 -\frac{2}{9} \pi^2 \log^3 2+\frac{4}{15} \log^5 2.
\end{equation}

The outline of this paper is as follows. In Section \ref{sec:bdconjecture}, we will introduce the Bombieri-Dwork conjecture. In Section \ref{sec:serieszeta3}, we will study a series representation of $\zeta(3)$ and construct a $G$-function from it. We will numerically evaluate the expansion of this $G$-function with respect to a canonical basis. In Section \ref{sec:gammaconjecture}, we will formulate a Gamma conjecture for the $G$-functions. In Section \ref{sec:furtherexamples}, we will provide more examples to support the Gamma conjecture formulated in Section \ref{sec:gammaconjecture}. In Section \ref{sec:furtherprospects}, we will discuss some future prospects and questions.

\section{The Bombieri-Dwork conjecture} \label{sec:bdconjecture}

The Bombieri-Dwork conjecture is an attempt to characterize which differential equations are of geometric origin. Before we introduce this conjecture, let us first look at an interesting example. The Legendre family of elliptic curves is defined by the equation
\begin{equation} \label{eq:legendrefamilyequation}
E_{\lambda}: y^2=x(x-1)(x-\lambda), ~\lambda \in \mathbb{C}- \{0,1\}.
\end{equation}
While when $\lambda=0,1$, this family degenerates to a nodal cubic, whose singularity is a double point. Let $\Omega^1_{E_{\lambda}}$ be the sheaf of algebraic oneforms on $E_{\lambda}$, which is in fact trivial, i.e. there exists a nowhere vanishing oneform $\Omega$ that defines a trivialization of $\Omega^1_{E_{\lambda}}$. In the affine open subset where $y \neq 0$, $\Omega$ is given by
\begin{equation}
\Omega=\frac{dx}{2\pi y}.
\end{equation}
It is a straightforward exercise to check that $dx/(2\pi y)$ actually extends to a nonvanishing oneform on $E_\lambda$. The oneform $\Omega$ satisfies the following Picard-Fuchs equation
\begin{equation} \label{eq:picardfuchsequation}
\mathscr{D}_L \Omega=0,~\mathscr{D}_L:=\lambda(\lambda-1)\,\frac{d^2 }{d\lambda^2} +(2\lambda-1) \,\frac{d}{d \lambda}+\frac{1}{4}.
\end{equation}
The periods of $\Omega$ are given by its integrals over homology cycles, and one period is given by the integral  
\begin{equation} \label{eq:periodsintegration}
\pi_0(\lambda)=\frac{1}{\pi}\int_{1}^{\infty} \frac{dx}{\sqrt{x(x-1)(x-\lambda)}}.
\end{equation}
For convenience, let us introduce the Euler-Gauss hypergeometric function
\begin{equation}
F(a,b;c;\lambda)=\sum_{n=0}^\infty \frac{(a)_n(b)_n}{(c)_n} \frac{\lambda^n}{n!},~(a)_n:=a(a+1)\cdots (a+n-1).
\end{equation}
The power series $\pi_0(\lambda)$ is equal to
\begin{equation}
\pi_0(\lambda)=\sum_{n=0}^\infty \binom{-1/2}{n}^2 \lambda^n=F(1/2,1/2;1;\lambda),
\end{equation}
which also satisfies the Picard-Fuchs equation
\begin{equation} \label{eq:periodsPFequation}
\mathscr{D}_L \pi_0(\lambda)=0.
\end{equation}
From the identity
\begin{equation}
\binom{-1/2}{n}=(-1)^n4^{-n} \binom{2n}{n},
\end{equation}
$\pi_0(\lambda)$ can also be written as
\begin{equation}
\pi_0(\lambda)=\sum_{n=0}^\infty 16^{-n} \binom{2n}{n}^2 \lambda^n.
\end{equation}
Asymptotically, we have
\begin{equation}
\binom{2n}{n} \sim \pi^{-1/2}\cdot 4^n \cdot n^{-1/2},
\end{equation}
hence we conclude that $\pi_0(\lambda)$ converges on the unit disc $|\lambda|<1$. From its form, the power series $\pi_0(\lambda)$ belongs to a class of functions called $G$-functions.
\begin{definition}
A power series $f(\phi)=\sum_{n=0}^\infty a_n \phi^n$ is called a $G$-function if its coefficients $a_n$ are algebraic numbers and there exists a constant $C>0$ such that:
\begin{enumerate}
\item assuming we have chosen an embedding $\overline{\mathbb{Q}} \subset \mathbb{C}$, we have $|a_n| \leq C^n$.

\item there exists a sequence of integers $d_n$ ($d_n \leq C^n$), and $d_na_m$ is an algebraic integer for every $m \leq n$.

\item $f(\phi)$ satisfies a linear homogeneous differential equation with coefficients in the polynomial ring $\overline{\mathbb{Q}}[\phi]$. 
\end{enumerate} 
\end{definition}

In a non-rigorous sense, the Bombieri-Dwork conjecture predicts that the differential equation satisfied by a $G$-function comes from geometry, i.e. it is the Picard-Fuchs equation satisfied by the periods of a family of algebraic varieties defined over $\overline{\mathbb{Q}}$. The readers are referred to the books \cite{Andre, Bald} for more details of this conjecture. The power series $\pi_0(\lambda)$ of course comes from geometry, as it is a period of the Legendre family of elliptic curves, and the differential equation it satisfies is just the Picard-Fuchs equation of this family.

\begin{definition}
By abuse of notations, the Picard-Fuchs equation of a $G$-function $f(\phi)$ is the lowest order linear homogeneous differential equation that it satisfies.
\end{definition}

\section{A series representation of \texorpdfstring{$\zeta(3)$}{z(3)}} \label{sec:serieszeta3}

In this paper, we will focus our attention on the $G$-functions $f$ such that $a_n \in \mathbb{Q}$ and the differential equation $f$ satisfies has a singular point with maximally unipotent monodromy. Let us first look at an interesting example about a series representation of $\zeta(3)$.

\subsection{The Picard-Fuchs equation}

In Ap\'ery's proof of the irrationality of $\zeta(3)$, the series representation
\begin{equation} \label{eq:zeta3powerseriessection}
\zeta(3)=\frac{5}{2}\, \sum_{n=1}^{\infty} \frac{(-1)^{n-1}}{n^3 \binom{2n}{n}}
\end{equation}
has played a crucial role. From it, we can construct a power series of the form
\begin{equation} \label{eq:integralperiodFirst}
\Pi_0(\phi)=2\,\sum_{n=1}^\infty  \frac{(-1)^{n-1}}{n^3 \binom{2n}{n}} \,\phi^{n-1}=1-\frac{\phi}{24}+\frac{\phi^2}{270}+\cdots.
\end{equation}
Furthermore, asymptotically we have
\begin{equation}
\left|\frac{(-1)^{n-1}}{n^3 \binom{2n}{n}} \right| \sim \pi^{1/2}\cdot \frac{1}{4^n}\cdot \frac{1}{n^{5/2}},
\end{equation}
from which we deduce that $\Pi_0(\phi)$ converges on the disc $|\phi| \leq 4$.
\begin{lemma}
The power series $\Pi_0(\phi)$ satisfies the following ODE
\begin{equation} \label{eq:picardfuchsphisection}
\left( \left(4+\phi \right)\vartheta^4+\left(10+4\phi \right) \vartheta^3+\left(8+6\phi \right) \vartheta^2+\left(2+4\phi \right)\vartheta+\phi \right) \Pi_0=0,~\vartheta=\phi \frac{d}{d\phi}.
\end{equation}
\end{lemma}
\begin{proof}
Suppose $f_0=1+\sum_{n=1}^\infty a_n\, \phi^n$ is a solution, then we plug it into this ODE \ref{eq:picardfuchsphisection}, and we get a recursion equation about $a_n$ with initial condition
\begin{equation}
1+24\,a_1=0,~n^3a_{n-1}+(n+1)(2n+1)(2n+2)a_n=0.
\end{equation}
It is straightforward to check that the coefficients of $\Pi_0(\phi)$ is a solution to this recursion equation, therefore we deduce that $\Pi_0(\phi)$ satisfies ODE \ref{eq:picardfuchsphisection}.
\end{proof}

For simplicity, let us define the Picard-Fuchs operator $\mathscr{D}_3$ by 
\begin{equation}
\mathscr{D}_3:=\left(4+\phi \right)\vartheta^4+\left(10+4\phi \right) \vartheta^3+\left(8+6\phi \right) \vartheta^2+\left(2+4\phi \right)\vartheta+\phi,~\vartheta=\phi \frac{d}{d\phi},
\end{equation}
which has three regular singularities at $\phi=0,-4,\infty$. $\Pi_0(\phi)$ can be analytically extended to a multi-valued holomorphic function on $\mathbb{C}-\{0,-4 \}$.

\subsection{The large complex structure limit}

Now, we define a new variable $\varphi$ by 
\begin{equation}
\varphi:=1/\phi,
\end{equation}
with respect to which, the differential operator $\mathscr{D}_3$ (up to a multiplication by $\varphi$) becomes
\begin{equation}
\mathscr{D}_3=\varphi^4(1+4 \varphi)\frac{d^4}{d\varphi^4}+2\varphi^3(1+7\varphi)\frac{d^3}{d\varphi^3}+\varphi^2(1+6\varphi)\frac{d^2}{d\varphi^2}-\varphi \frac{d}{d\varphi}+1
\end{equation}
There exist four canonical solutions to $\mathscr{D}_3$ of the form
\begin{equation} \label{eq:canonicalperiodssection1}
\begin{aligned}
\varpi_0(\varphi)&=\varphi, \\
\varpi_1(\varphi)&=\varphi \log \varphi,\\
\varpi_2(\varphi)&=\varphi \log^2 \varphi+h_2(\varphi),\\
\varpi_3(\varphi)&=\varphi \log^3 \varphi+3h_2(\varphi) \log \varphi+h_3(\varphi),
\end{aligned}
\end{equation}
where $h_2(\varphi)$ and $h_3(\varphi)$ are power series in $\varphi$. $h_2(\varphi)$ is of the form
\begin{equation}
h_2(\varphi)=\sum_{n=2}^\infty b_n \varphi^n,
\end{equation}
whose coefficients are determined by the recursion equation
\begin{equation}
b_2=-4,~(n-1)(2n-1)(2n-2)b_n+n^3b_{n+1}=0.
\end{equation}
This recursion equation can be solved explicitly and we obtain
\begin{equation}
b_{n+1}=\frac{2!(-1)^{n}\binom{2n}{n}}{n^2}.
\end{equation}
When $n$ goes to infinity, the absolute value of $b_n$ is asymptotically given by
\begin{equation}
|b_n| \sim \frac{2}{ \pi^{1/2}}\cdot 4^n \cdot \frac{1}{n^{5/2}},
\end{equation}
therefore $h_2(\varphi)$ converges on the disc $|\varphi| \leq 1/4$. While the power series $h_3(\varphi)$
\begin{equation}
h_3(\varphi)=\sum_{n=2}^\infty c_n \varphi^n
\end{equation}
is determined by the recursion equation
\begin{equation}
c_2=12,~6(3-8n+5n^2)b_n+2n(n-1)^2(2n-1)c_n+n^4c_{n+1}=0.
\end{equation}
This recursion relation can also be explicitly solved, and we have
\begin{equation}
c_{n+1}=\frac{3!(-1)^{n+1} \binom{2n}{n}}{n^3}\left(3+2n(H_{n-1}-H_{2n-1} ) \right),
\end{equation}
where $H_r$ is the harmonic number
\begin{equation}
H_r:=\sum_{k=1}^r 1/k.
\end{equation} 
When $n$ goes to infinity, the absolute value of $c_n$ is asymptotically given by
\begin{equation}
|c_n| \sim \frac{12 \log 2}{\pi^{1/2}}\cdot 4^n \cdot \frac{1}{n^{5/2}},
\end{equation}
therefore $h_3(\varphi)$ converges on the disc $|\varphi| \leq 1/4$.

\subsection{\texorpdfstring{$\zeta(3)$}{z(3)} and periods} \label{sec:zeta3andcanonicalperiods}

The four canonical solutions $\{\varpi_i(\varphi) \}_{i=0}^3$ form a basis of the solution space of the Picard-Fuchs operator $\mathscr{D}_3$. Since $\Pi_0(\phi)$ is also a solution of $\mathscr{D}_3$, there must exist an expansion of $\Pi_0(\phi)$ with respect to the basis $\{\varpi_i(\varphi) \}_{i=0}^3$ of the form
\begin{equation}
\Pi_0(\phi)=\tau_0\varpi_0(\varphi)+\tau_1\varpi_1(\varphi)+\tau_2\varpi_2(\varphi)+\tau_3\varpi_3(\varphi),~\tau_i \in \mathbb{C}.
\end{equation}
However, it is very difficult to compute the numbers $\tau_i$ analytically, hence in this paper we will resort to a numerical method \cite{Yangapery}. From the formula \ref{eq:integralperiodFirst},  we can numerically compute the values of $\Pi_0(\phi)$ at the four points
\begin{equation}
\phi=1/2,1/3,1/4,1/5
\end{equation}
very efficiently. Next, we numerically evaluate $\varpi_i(\varphi)$ at the points
\begin{equation}
\varphi=2,3,4,5.
\end{equation}
However the power series $h_2(\varphi)$ and $h_3(\varphi)$ only converges on the disc $|\varphi|\leq 1/4$, thus we will use a different method to evaluate them at the above points. As an example, we will explain how to numerically compute the value $\varpi_2(2)$. From the power series expansion of $h_2(\varphi)$, we can numerically compute the values of $\varpi_2(\varphi)$ and its derivatives at $\varphi=1/20$ to a very high precision. Then we can numerically solve the Picard-Fuchs equation $\mathscr{D}_3$ on the interval $[1/20,2]$, from which we obtain the value of $\varpi_2(\varphi)$ at $\varphi=2$. Our Mathematica program has numerically shown that the values of $\tau_i$ agree with
\begin{equation} \label{eq:valuesoftaui}
\tau_0=4\,\zeta(3),~\tau_1=0,~\tau_2=0,~\tau_3=-\frac{1}{3}.
\end{equation}
We can easily check it to the first 100 digits. 
\begin{remark}
The occurrence of $\zeta(3)$ in $\tau_0$ is very interesting, and it seems to have close connection with mirror symmetry. The readers are referred to the paper \cite{KimYang} for more details.
\end{remark}

At the points $\varphi=\pm 1/4$ ($\phi=\pm 4$), the power series $\Pi_0(\phi)$, $h_2(\varphi)$ and $h_3(\varphi)$ all converge, which yields interesting identities
\begin{equation}
\Pi_0(4)=4\,\zeta(3)\, \varpi_0(1/4)-\frac{1}{3} \,\varpi_3(1/4),~\Pi_0(-4)=4\,\zeta(3)\, \varpi_0(-1/4)-\frac{1}{3} \,\varpi_3(-1/4).
\end{equation}
Let us now look at the two identities in detail. The first one implies
\begin{equation}
\begin{aligned}
2\,\sum_{n=1}^\infty  \frac{(-1)^{n-1}}{n^3 \binom{2n}{n}} \,4^{n-1}=&\zeta(3)+\frac{2}{3}\log^3 2  +  \sum_{n=1}^\infty \frac{(-1)^{n}\binom{2n}{n}}{4^n n^2} \log 2 \\
&-\frac{1}{2}\sum_{n=1}^\infty \frac{(-1)^{n+1} \binom{2n}{n}}{ n^3 4^n}\left(3+2n(H_{n-1}-H_{2n-1} ) \right).
\end{aligned}
\end{equation}
While the second one gives us two identities
\begin{equation}
\begin{aligned}
2 \sum_{n=1}^\infty  \frac{4^{n-1}}{n^3 \binom{2n}{n}} &=-\zeta(3)+\frac{1}{3}\pi^2 \log 2+\frac{4}{3}\log^3 2 
-\frac{1}{2}\sum_{n=1}^\infty \frac{\binom{2n}{n}}{ 4^n n^3}\left(3+2n(H_{n-1}-H_{2n-1} ) \right),\\
\sum_{n=1}^\infty \frac{\binom{2n}{n}}{4^n n^2}&=\frac{1}{6} \pi^2-2\,\log^2 2.
\end{aligned}
\end{equation}

\section{The Gamma conjecture for \texorpdfstring{$G$}{G}-functions} \label{sec:gammaconjecture}

In this section, we will formulate a Gamma conjecture for the $G$-functions whose differential equations have maximally unipotent monodromy at a singular point. Suppose $f(\phi)=\sum_{n=0}^\infty a_n \phi^n$ is a $G$-function, and for simplicity, we will assume $a_n \in \mathbb{Q}$ from now on. Suppose $\phi_0$ is a singularity of $f(\phi)$'s Picard-Fuchs equation, denoted by $\mathscr{D}_f$, where the monodromy is maximally unipotent. We define the variable $\varphi$ by
\begin{equation}
\varphi=
\begin{cases}
1/\phi,\,\,\,\,\,\,\,\,\,\text{if}~\phi_0=\infty;\\
\phi-\phi_0,~\text{otherwise}.\\
\end{cases}
\end{equation}
Suppose the order of $\mathscr{D}_f$ is $n$. Since the monodromy of $\mathscr{D}_f$ at $\phi_0$ is maximally unipotent, in a small neighborhood of $\varphi=0$, there exists a basis of the solution space given by
\begin{equation}
\varpi_j(\varphi)=\frac{1}{(2 \pi i)^j}\sum_{k=0}^{j} \binom{j}{k} h_k(\varphi)\,\log^{j-k} \varphi,~j=0,1\cdots n-1;
\end{equation}
where $h_k(\varphi)$ are power series in $\varphi$ that satisfy
\begin{equation}
h_k(0)=0.
\end{equation}
Since $f(\phi)$ is also a solution of $\mathscr{D}_f$, there exist complex numbers $\tau_i^f$ such that
\begin{equation}
f(\phi)=\tau^f_0 \varpi_0(\varphi)+\tau^f_1 \varpi_1(\varphi)+\cdots+\tau^f_{n-1} \varpi_{n-1}(\varphi).
\end{equation}
We now formulate a Gamma conjecture for the $G$-function $f(\phi)$ \cite{Galkin,KimYang,Yang}.
\vspace*{0.1in}

\textbf{Conjecture G}: \textit{After a rescaling of the variable $\varphi$ by $\varphi \rightarrow k \varphi,k \in \mathbb{Q}$, the complex numbers $\tau^f_i$ lie in the ring over $\mathbb{Q}$ generated by $2 \pi i$ and multiple zeta values}.

\vspace*{0.1in}

On the other hand, if we assume the Bombieri-Dwork conjecture, the Picard-Fuchs equation $\mathscr{D}_f$ comes from geometry. Then \textbf{Conjecture G} has close connections with the Gamma conjecture in algebraic geometry \cite{Galkin,KimYang,Yang}. It is certainly very interesting to provide more examples for this conjecture, e.g. the papers \cite{Yangdoublezeta, Yangapery}. Now we will provide more examples in next section.

\section{Further examples} \label{sec:furtherexamples}

In this section, we will look at the series
\begin{equation}
\sum_{n=1}^{\infty} \frac{(-1)^{n-1}}{n^k \binom{2n}{n}},~ k \geq 2, k \in \mathbb{Z}_+.
\end{equation}
By abuse of notation, we again construct a power series $\Pi_0(\phi)$ by
\begin{equation}
\Pi_0(\phi)=2\,\sum_{n=1}^\infty  \frac{(-1)^{n-1}}{n^k \binom{2n}{n}} \,\phi^{n-1},
\end{equation}
which converges on the disc $|\phi| \leq 4$. Let us now study the first several small $k$.

\subsection{\texorpdfstring{$k=2$}{k=2}}

When $k=2$, $\Pi_0(\phi)$ is a solution to the following ODE
\begin{equation}
\mathscr{D}_2\Pi_0(\phi)=0,~\mathscr{D}_2:= \left(4+\phi \right)\vartheta^3+\left(6+3\phi \right) \vartheta^2+\left(2+3\phi \right)\vartheta+\phi.
\end{equation}
In a neighborhood of $\varphi=0$, $\mathscr{D}_2$ has three canonical solutions given by
\begin{equation}
\begin{aligned}
\varpi_0(\varphi)&=\varphi, \\
\varpi_1(\varphi)&=\varphi \log \varphi+ \sum_{n=1}^\infty \frac{(-1)^n\binom{2n}{n}}{n} \varphi^{n+1},\\
\varpi_2(\varphi)&=\varphi \log^2 \varphi+2 \left( \sum_{n=1}^\infty \frac{(-1)^n\binom{2n}{n}}{n} \varphi^{n+1} \right) \log \varphi \\
&~~~+\sum_{n=1}^\infty \frac{2!(-1)^{n+1}\binom{2n}{n}}{n^2}\left( 2+2n(H_{n-1}-H_{2n-1}) \right)\varphi^{n+1}.
\end{aligned}
\end{equation}
Numerically, we have found that the expansion of $\Pi_0(\phi)$ with respect to this basis is 
\begin{equation}
\Pi_0(\phi)=\varpi_2(\varphi).
\end{equation}
The equation
\begin{equation}
\Pi_0(4)=\varpi_2(1/4)
\end{equation}
yields the following identity
\begin{equation}
\begin{aligned}
2\,\sum_{n=1}^\infty  \frac{(-1)^{n-1}}{n^2 \binom{2n}{n}} \,4^{n-1}&=\log^2 2- \left( \sum_{n=1}^\infty \frac{(-1)^n\binom{2n}{n}}{4^n n} \right) \log 2 \\
&~~~+\frac{1}{2}\sum_{n=1}^\infty \frac{(-1)^{n+1}\binom{2n}{n}}{4^{n} n^2}\left( 2+2n(H_{n-1}-H_{2n-1}) \right).
\end{aligned}
\end{equation}
Since $\Pi_0(-4)$ is a real number, the equation
\begin{equation}
\Pi_0(-4)=\varpi_2(-1/4)
\end{equation}
actually yields two identities
\begin{equation}
\Pi_0(-4)=\text{Re}\,\varpi_2(-1/4),~\text{Im}\,\varpi_2(-1/4)=0;
\end{equation}
which respectively are
\begin{equation}
\begin{aligned}
2\,\sum_{n=1}^\infty  \frac{4^{n-1}}{n^2 \binom{2n}{n}} &=\frac{1}{4} \pi^2+\log^2 2+\frac{1}{2}\sum_{n=1}^\infty \frac{ \binom{2n}{n}}{4^{n} n^2}\left( 2+2n(H_{n-1}-H_{2n-1}) \right), \\
\sum_{n=1}^\infty \frac{\binom{2n}{n}}{4^n n}&=2 \log 2.
\end{aligned}
\end{equation}

\subsection{\texorpdfstring{$k=4$}{k=4}}

When $k=4$, $\Pi_0(\phi)$ is a solution to the differential operator
\begin{equation}
\mathscr{D}_4:= \left(4+\phi \right)\vartheta^5+\left(14+5\phi \right) \vartheta^4+(18+10 \phi)\vartheta^3+\left(10+10\phi \right)\vartheta^2+(2+5\phi) \vartheta+\phi.
\end{equation}
In a neighborhood of $\varphi=0$, $\mathscr{D}_4$ has five canonical solutions given by
\begin{equation}
\begin{aligned}
\varpi_i&=\varphi \,\log^i \,\varphi,~i=0,1,2, \\
\varpi_3&=\varphi \log^3 \varphi+ \sum_{n=1}^\infty \frac{3! (-1)^n\binom{2n}{n}}{n^3} \varphi^{n+1},\\
\varpi_4&=\varphi \log^4 \varphi+4 \left( \sum_{n=1}^\infty \frac{3!(-1)^n\binom{2n}{n}}{n^3} \varphi^{n+1} \right) \log \varphi \\
&~~~+\sum_{n=1}^\infty \frac{4! (-1)^{n+1}\binom{2n}{n}}{n^4}\left( 4+2n(H_{n-1}-H_{2n-1}) \right)\varphi^{n+1}.
\end{aligned}
\end{equation}
Numerically, we have found that the expansion of $\Pi_0(\phi)$ with respect to this basis is 
\begin{equation}
\Pi_0(\phi)=-6\,\zeta(4)\,\varpi_0(\varphi)-4\,\zeta(3)\,\varpi_1(\varphi)+\frac{1}{12}\varpi_4(\varphi).
\end{equation}
As before, the equation
\begin{equation}
\Pi_0(4)=-6\,\zeta(4)\,\varpi_0(1/4)-4\,\zeta(3)\,\varpi_1(1/4)+\frac{1}{12}\varpi_4(1/4).
\end{equation}
yields the following identity
\begin{equation}
\begin{aligned}
2\,\sum_{n=1}^\infty  \frac{(-1)^{n-1}}{n^4 \binom{2n}{n}} \,4^{n-1}&=-\frac{1}{60}\pi^4+2 \zeta(3) \log 2+\frac{1}{3} \log^4 2 -\left( \sum_{n=1}^\infty \frac{(-1)^n\binom{2n}{n}}{4^n n^3} \right) \log 2 \\
&~~~+\frac{1}{2}\sum_{n=1}^\infty \frac{(-1)^{n+1}\binom{2n}{n}}{4^{n} n^4}\left( 4+2n(H_{n-1}-H_{2n-1}) \right).
\end{aligned}
\end{equation}
Since $\Pi_0(-4)$ is a real number, the real part and imaginary part of the equation
\begin{equation}
\Pi_0(-4)=-6\,\zeta(4)\,\varpi_0(-1/4)-4\,\zeta(3)\,\varpi_1(-1/4)+\frac{1}{12}\varpi_4(-1/4).
\end{equation}
give us two identities
\begin{equation}
\begin{aligned}
2\,\sum_{n=1}^\infty  \frac{4^{n-1}}{n^4 \binom{2n}{n}} &=-\frac{1}{240} \pi^4+\frac{1}{6}\pi^2 \log^2 2+\log^4 2+\frac{1}{2}\sum_{n=1}^\infty \frac{ \binom{2n}{n}}{4^{n} n^4}\left( 4+2n(H_{n-1}-H_{2n-1}) \right), \\
\sum_{n=1}^\infty \frac{\binom{2n}{n}}{4^n n^3}&=2\, \zeta(3)-\frac{1}{3} \pi^2 \log 2+\frac{4}{3} \log^3 2.
\end{aligned}
\end{equation}

\subsection{\texorpdfstring{$k=5$}{k=5}}

When $k=5$, $\Pi_0(\phi)$ is a solution to the differential operator
\begin{equation}
\begin{aligned}
\mathscr{D}_5:=& \left(4+\phi \right)\vartheta^6+\left(18+6\phi \right) \vartheta^5+(32+15 \phi)\vartheta^4, \\
&+\left(28+20\phi \right)\vartheta^3+(12+15\phi) \vartheta^2+(2+6 \phi)\vartheta+\phi.
\end{aligned}
\end{equation}
In a neighborhood of $\varphi=0$, $\mathscr{D}_5$ has six canonical solutions given by
\begin{equation}
\begin{aligned}
\varpi_i&=\varphi \,\log^i \,\varphi,~i=0,1,2,3, \\
\varpi_4&=\varphi \log^4 \varphi+ \sum_{n=1}^\infty \frac{4! (-1)^n\binom{2n}{n}}{n^4} \varphi^{n+1},\\
\varpi_5&=\varphi \log^5 \varphi+5 \left( \sum_{n=1}^\infty \frac{4! (-1)^n\binom{2n}{n}}{n^4} \varphi^{n+1} \right) \log \varphi \\
&~~~+\sum_{n=1}^\infty \frac{5! (-1)^{n+1}\binom{2n}{n}}{n^5}\left( 5+2n(H_{n-1}-H_{2n-1}) \right)\varphi^{n+1}.
\end{aligned}
\end{equation}
Numerically, we have found that the expansion of $\Pi_0(\phi)$ with respect to this basis is 
\begin{equation}
\Pi_0(\phi)=12\,\zeta(5)\,\varpi_0(\varphi)+6\,\zeta(4)\,\varpi_1(\varphi)+2\,\zeta(3)\,\varpi_2(\varphi)-\frac{1}{60}\varpi_5(\varphi).
\end{equation}
As before, the equation
\begin{equation}
\Pi_0(4)=12\,\zeta(5)\,\varpi_0(1/4)+6\,\zeta(4)\,\varpi_1(1/4)+2\,\zeta(3)\,\varpi_2(1/4)-\frac{1}{60}\varpi_5(1/4).
\end{equation}
yields the following identity
\begin{equation}
\begin{aligned}
2\,\sum_{n=1}^\infty  \frac{(-1)^{n-1}}{n^5 \binom{2n}{n}} \,4^{n-1}&=3\,\zeta(5)-\frac{1}{30}\pi^4 \log 2+2 \zeta(3) \log^2 2+\frac{2}{15} \log^5 2 +\left( \sum_{n=1}^\infty \frac{(-1)^n\binom{2n}{n}}{4^n n^4} \right) \log 2 \\
&~~~+\frac{1}{2}\sum_{n=1}^\infty \frac{(-1)^{n}\binom{2n}{n}}{4^{n} n^5}\left( 5+2n(H_{n-1}-H_{2n-1}) \right).
\end{aligned}
\end{equation}
Since $\Pi_0(-4)$ is a real number, the real part and the imaginary part of the equation
\begin{equation}
\Pi_0(-4)=12\,\zeta(5)\,\varpi_0(-1/4)+6\,\zeta(4)\,\varpi_1(-1/4)+2\,\zeta(3)\,\varpi_2(-1/4)-\frac{1}{60}\varpi_5(-1/4).
\end{equation}
give us two identities
\begin{equation}
\begin{aligned}
2\,\sum_{n=1}^\infty  \frac{4^{n-1}}{n^5 \binom{2n}{n}} &=-3 \zeta(5)-\frac{1}{30} \pi^4 \log 2+2\zeta(3) \log^2 2+\frac{1}{2} \zeta(3) \pi^2+\frac{8}{15} \log^5 2 \\
&-\frac{1}{2} \sum_{n=1}^\infty \frac{ \binom{2n}{n}}{4^{n} n^5}\left( 5+2n(H_{n-1}-H_{2n-1}) \right), \\
\sum_{n=1}^\infty \frac{\binom{2n}{n}}{4^n n^4}&=\frac{1}{40} \pi^4-4  \zeta(3) \log 2+\frac{1}{3} \pi^2 \log^2 2-\frac{2}{3} \log^4 2.
\end{aligned}
\end{equation}

\subsection{\texorpdfstring{$k=6$}{k=6}}

When $k=6$, $\Pi_0(\phi)$ is a solution to the differential operator
\begin{equation}
\begin{aligned}
\mathscr{D}_6:=& \left(4+\phi \right)\vartheta^7+\left(22+7 \phi \right) \vartheta^6+\left(50+21\phi \right) \vartheta^5+(60+35 \phi)\vartheta^4, \\
&+\left(40+35\phi \right)\vartheta^3+(14+21\phi) \vartheta^2+(2+7 \phi)\vartheta+\phi.
\end{aligned}
\end{equation}
In a neighborhood of $\varphi=0$, $\mathscr{D}_6$ has seven canonical solutions given by
\begin{equation}
\begin{aligned}
\varpi_i&=\varphi \,\log^i \,\varphi,~i=0,1,2,3,4, \\
\varpi_5&=\varphi \log^5 \varphi+ \sum_{n=1}^\infty \frac{5! (-1)^n\binom{2n}{n}}{n^5} \varphi^{n+1},\\
\varpi_6&=\varphi \log^6 \varphi+6 \left( \sum_{n=1}^\infty \frac{5! (-1)^n\binom{2n}{n}}{n^5} \varphi^{n+1} \right) \log \varphi \\
&~~~+\sum_{n=1}^\infty \frac{6! (-1)^{n+1}\binom{2n}{n}}{n^6}\left( 6+2n(H_{n-1}-H_{2n-1}) \right)\varphi^{n+1}.
\end{aligned}
\end{equation}
Numerically, we have found that the expansion of $\Pi_0(\phi)$ with respect to this basis is 
\begin{equation}
\begin{aligned}
\Pi_0(\phi)=&(-20\,\zeta(6)+4\,\zeta(3)^2)\varpi_0(\varphi)-12\,\zeta(5)\,\varpi_1(\varphi) \\
&-3\,\zeta(4)\,\varpi_2(\varphi)-\frac{2}{3}\,\zeta(3)\,\varpi_3(\varphi)+\frac{2}{6!}\varpi_6(\varphi).
\end{aligned}
\end{equation}
As before, the equation
\begin{equation}
\begin{aligned}
\Pi_0(4)=&(-20\,\zeta(6)+4\,\zeta(3)^2)\varpi_0(1/4)-12\,\zeta(5)\,\varpi_1(1/4) \\
&-3\,\zeta(4)\,\varpi_2(1/4)-\frac{2}{3}\,\zeta(3)\,\varpi_3(1/4)+\frac{2}{6!}\varpi_6(1/4).
\end{aligned}
\end{equation}
yields the following identity
\begin{equation}
\begin{aligned}
2\,\sum_{n=1}^\infty  \frac{(-1)^{n-1}}{n^6 \binom{2n}{n}} \,4^{n-1}&=-\frac{1}{189} \pi^6+6\,\zeta(5)\log 2-\frac{1}{30}\pi^4 \log^2 2+\zeta(3)^2+\frac{4}{3} \zeta(3) \log^3 2+\frac{2}{45} \log^6 2 \\
&-\left( \sum_{n=1}^\infty \frac{(-1)^n\binom{2n}{n}}{4^n n^5} \right) \log 2 +\frac{1}{2}\sum_{n=1}^\infty \frac{(-1)^{n+1}\binom{2n}{n}}{4^{n} n^6}\left( 6+2n(H_{n-1}-H_{2n-1}) \right).
\end{aligned}
\end{equation}
Since $\Pi_0(-4)$ is a real number, the real part and the imaginary part of the equation
\begin{equation}
\begin{aligned}
\Pi_0(-4)=&(-20\,\zeta(6)+4\,\zeta(3)^2)\varpi_0(-1/4)-12\,\zeta(5)\,\varpi_1(-1/4) \\
&-3\,\zeta(4)\,\varpi_2(-1/4)-\frac{2}{3}\,\zeta(3)\,\varpi_3(-1/4)+\frac{2}{6!}\varpi_6(-1/4).
\end{aligned}
\end{equation}
give us two identities
\begin{equation}
\begin{aligned}
2\,\sum_{n=1}^\infty  \frac{4^{n-1}}{n^6 \binom{2n}{n}} &=-\frac{71}{30240} \pi^6-\frac{7}{120} \pi^4 \log^2 2-\zeta(3)^2+\frac{2}{3}\zeta(3)\pi^2\log 2+\frac{8}{3} \zeta(3) \log^3 2\\
&-\frac{1}{18}\pi^2 \log^4 2+\frac{2}{9} \log^6 2 +\frac{1}{2} \sum_{n=1}^\infty \frac{ \binom{2n}{n}}{4^{n} n^6}\left( 6+2n(H_{n-1}-H_{2n-1}) \right), \\
\sum_{n=1}^\infty \frac{\binom{2n}{n}}{4^n n^5}&=6\,\zeta(5)-\frac{1}{20} \pi^4 \log 2-\frac{1}{3} \zeta(3) \pi^2+4 \, \zeta(3) \, \log^2 2 -\frac{2}{9} \pi^2 \log^3 2+\frac{4}{15} \log^5 2.
\end{aligned}
\end{equation}

\subsection{\texorpdfstring{$k=7$}{k=7}}

When $k=7$, $\Pi_0(\phi)$ is a solution to the differential operator
\begin{equation}
\begin{aligned}
\mathscr{D}_7:=& \left(4+\phi \right)\vartheta^8+\left(26+8 \phi \right) \vartheta^7+\left(72+28 \phi \right) \vartheta^6+\left(110+56 \phi \right) \vartheta^5+(100+70 \phi)\vartheta^4, \\
&+\left(54+56\phi \right)\vartheta^3+(16+28\phi) \vartheta^2+(2+8 \phi)\vartheta+\phi.
\end{aligned}
\end{equation}
In a neighborhood of $\varphi=0$, $\mathscr{D}_7$ has eight canonical solutions given by
\begin{equation}
\begin{aligned}
\varpi_i&=\varphi \,\log^i \,\varphi,~i=0,1,2,3,4,5, \\
\varpi_6&=\varphi \log^6 \varphi+ \sum_{n=1}^\infty \frac{6! (-1)^n\binom{2n}{n}}{n^6} \varphi^{n+1},\\
\varpi_7&=\varphi \log^7 \varphi+7 \left( \sum_{n=1}^\infty \frac{6! (-1)^n\binom{2n}{n}}{n^6} \varphi^{n+1} \right) \log \varphi \\
&~~~+\sum_{n=1}^\infty \frac{7! (-1)^{n+1}\binom{2n}{n}}{n^7}\left( 7+2n(H_{n-1}-H_{2n-1}) \right)\varphi^{n+1}.
\end{aligned}
\end{equation}
Numerically, we have found that the expansion of $\Pi_0(\phi)$ with respect to this basis is 
\begin{equation}
\begin{aligned}
\Pi_0(\phi)=&(36\zeta(7)-12 \zeta(3)\zeta(4) ) \varpi_0(\varphi)+(20\,\zeta(6)-4\,\zeta(3)^2)\varpi_1(\varphi)+6\,\zeta(5)\,\varpi_2(\varphi) \\
&+\,\zeta(4)\,\varpi_3(\varphi)+\frac{1}{6}\,\zeta(3)\,\varpi_4(\varphi)-\frac{2}{7!}\varpi_7(\varphi).
\end{aligned}
\end{equation}
As before, the equation
\begin{equation}
\begin{aligned}
\Pi_0(4)=&(36\zeta(7)-12 \zeta(3)\zeta(4) ) \varpi_0(1/4)+(20\,\zeta(6)-4\,\zeta(3)^2)\varpi_1(1/4)+6\,\zeta(5)\,\varpi_2(1/4) \\
&+\,\zeta(4)\,\varpi_3(1/4)+\frac{1}{6}\,\zeta(3)\,\varpi_4(1/4)-\frac{2}{7!}\varpi_7(1/4).
\end{aligned}
\end{equation}
yields the following identity
\begin{equation}
\begin{aligned}
2\,\sum_{n=1}^\infty  \frac{(-1)^{n-1}}{n^7 \binom{2n}{n}} \,4^{n-1}&=9\zeta(7)-\frac{2}{189} \pi^6 \log 2+6\,\zeta(5)\log^2 2-\frac{1}{45}\pi^4 \log^3 2-\frac{1}{30} \pi^4 \zeta(3) \\
&~~~+\frac{2}{3}\zeta(3) \log^4 2+2 \zeta(3)^2 \log 2+\frac{4}{315} \log^7 2+\left( \sum_{n=1}^\infty \frac{(-1)^n\binom{2n}{n}}{4^n n^6} \right) \log 2 \\
&~~~-\frac{1}{2}\sum_{n=1}^\infty \frac{(-1)^{n+1}\binom{2n}{n}}{4^{n} n^7}\left( 7+2n(H_{n-1}-H_{2n-1}) \right).
\end{aligned}
\end{equation}
Since $\Pi_0(-4)$ is a real number, the real part and the imaginary part of the equation
\begin{equation}
\begin{aligned}
\Pi_0(-4)=&(36\zeta(7)-12 \zeta(3)\zeta(4) ) \varpi_0(-1/4)+(20\,\zeta(6)-4\,\zeta(3)^2)\varpi_1(-1/4)+6\,\zeta(5)\,\varpi_2(-1/4) \\
&+\,\zeta(4)\,\varpi_3(-1/4)+\frac{1}{6}\,\zeta(3)\,\varpi_4(-1/4)-\frac{2}{7!}\varpi_7(-1/4).
\end{aligned}
\end{equation}
give us two identities
\begin{equation}
\begin{aligned}
2\,\sum_{n=1}^\infty  \frac{4^{n-1}}{n^7 \binom{2n}{n}} &=-9\zeta(7)-\frac{5}{504} \pi^6 \log 2 +\frac{3}{2} \zeta(5) \pi^2+6 \zeta(5) \log^2 2-\frac{1}{18} \pi^4 \log^3 2\\
&-\frac{1}{120} \pi^4 \zeta(3) +\frac{1}{3}\zeta(3) \pi^2 \log^2 2+2 \zeta(3) \log^4 2-\frac{2}{45} \pi^2 \log^5 2+\frac{8}{105} \log^7 2\\
&-\frac{1}{2} \sum_{n=1}^\infty \frac{ \binom{2n}{n}}{4^{n} n^7}\left( 7+2n(H_{n-1}-H_{2n-1}) \right), \\
\sum_{n=1}^\infty \frac{\binom{2n}{n}}{4^n n^6}&=\frac{79}{15120} \pi^6-12\,\zeta(5) \log 2+\frac{1}{20} \pi^4 \log^2 2-2 \zeta(3)^2\\
&~~~+\frac{2}{3} \zeta(3) \pi^2 \log 2-\frac{8}{3} \, \zeta(3) \, \log^3 2 +\frac{1}{9} \pi^2 \log^4 2-\frac{4}{45} \log^6 2.
\end{aligned}
\end{equation}

\section{Future prospects and comments} \label{sec:furtherprospects}

We will loosely end this paper with several small remarks and open questions. The Bombieri-Dwork conjecture is of course very difficult to prove, and not much is known by now. But the \textbf{Conjecture G} formulated in this paper seems to be more accessible. Moreover, \textbf{Conjecture G} has very close connections to the mirror symmetry of Calabi-Yau manifolds and the Gamma conjecture in algebraic geometry. On the other hand, we have also provided a numerical method to check \textbf{Conjecture G}, which also gives us interesting infinite series sums about zeta values. It is certainly very interesting to provide more examples to \textbf{Conjecture G}.

\section*{Acknowledgments}

The author is grateful to Minhyong Kim for a reading of the draft and many helpful comments.

\appendix

\end{document}